\documentclass[leqno,10pt]{amsart}

\usepackage[english]{babel}
\usepackage[T1]{fontenc}

\usepackage{amsmath,amssymb,amsthm}
\usepackage{mathtools}  
\usepackage[margin=1.25in]{geometry}

\usepackage{graphics,tikz,caption,subcaption}

\usepackage{mathrsfs}
\usepackage{stmaryrd}

\usepackage[colorlinks	=	true, linkcolor	=	red, urlcolor	=	red, citecolor	=	red]{hyperref}
\usepackage{bookmark}									
\usepackage{cleveref}

\usepackage{multicol}
\usepackage{todonotes}
\usepackage{enumitem}
\usepackage{verbatim}

\theoremstyle{definition}
\newtheorem{thm}{Theorem}[section]

\newtheorem{lemm}[thm]{Lemma}

\newtheorem{prop}[thm]{Proposition}
\newtheorem{ques}[thm]{Question}

\newtheoremstyle{case}
{3pt}
{3pt}
{}
{}
{\itshape}
{:}
{.5em}
{}

\theoremstyle{case} 
\newtheorem{case1}{Case}
\newtheorem{case2}{Case}

\theoremstyle{remark}

\DeclareMathOperator{\Mod}{mod}

\DeclareMathOperator{\diam}{diam}
\DeclareMathOperator{\dist}{dist}

\allowdisplaybreaks

\newcommand{\apmd}[2][]{															
	\ifthenelse{\equal{#1}{}}%
					{ \operatorname{N}_{#2}	}%
					{ \operatorname{N}_{#1}(#2) 	}}
					
\begin{document}

	\title{Remarks on conformal modulus in metric spaces}
	\author{Matthew Romney}
	\address{Mathematics Department, Stony Brook University, Stony Brook NY, 11794, USA.}
    \email{matthew.romney@stonybrook.edu}
	
	\date{\today}
	\maketitle
	
	\begin{abstract}
	    We give an example of an Ahlfors $3$-regular, linearly locally connected metric space homeomorphic to $\mathbb{R}^3$ containing a nondegenerate continuum $E$ with zero capacity, in the sense that the conformal modulus of the set of nontrivial curves intersecting $E$ is zero. We discuss this example in relation to the quasiconformal uniformization problem for metric spaces.  
	\end{abstract}

	
	

	\section{Introduction}
	
	The \textit{quasiconformal uniformization problem} asks one to determine which metric spaces can be mapped onto a standard model space, such as $\mathbb{R}^n$ or $\mathbb{S}^n$ for some $n\geq 2$, by a quasiconformal homeomorphism. It is inspired by the classical uniformization theorem, which states that every simply connected Riemann surface is conformally equivalent to either the unit disk, the plane or the $2$-sphere. In the metric space setting, quasiconformal homeomorphisms can be defined using the notion of \textit{conformal modulus} of a curve family, which measures the size of a curve family in a conformally invariant way. This notion provides the basis for to the analytic approach to the uniformization problem in dimension two due to Rajala \cite{Raj:17} and generalized in \cite{MW:21,NR:22,NR:21}. In particular, it is shown in \cite{NR:22} that any metric surface of locally finite Hausdorff $2$-measure admits a \textit{weakly quasiconformal} parametrization by a smooth Riemannian surface of constant curvature. See \Cref{sec:background} for additional discussion. By \textit{metric surface}, we mean a metric space homeomorphic to a $2$-manifold (with or without boundary). We use the term \textit{metric $n$-manifold} with a similar meaning. We assume that all manifolds are connected. 
	
	One reason for the effectiveness of conformal modulus methods in the two-dimensional setting is the abundance of rectifiable curves on any metric surface $X$ with locally finite Hausdorff $2$-measure. Indeed, a simple application of the co-area formula for metric spaces yields the following consequence: for any two disjoint nondegenerate continua $E,F \subset X$, the conformal modulus of the family of curves intersecting both $E$ and $F$ is positive. See Proposition 3.5 in \cite{Raj:17}. Another way to phrase this is that any continuum in $X$ has positive capacity. A quantitative version of this fact was shown by the author with K. Rajala \cite{RR:19} and subsequently sharpened and generalized by Eriksson-Bique and Poggi-Corraddini \cite{EBPC:21} (\Cref{thm:ebpc} below). We also refer the reader to Semmes \cite{Sem:96} for one of the seminal works on the topic of curves in metric spaces. 
	
	The purpose of this note is to observe that the above property of conformal modulus in the two-dimensional setting does not generalize to higher dimensions, even under the nice geometric assumptions of Ahlfors regularity and linear local connectivity. Given two disjoint continua $E,F$ in a metric space $X$, we let $\Gamma_E$ denote the family of nontrivial (i.e., nonconstant) curves intersecting $E$ and $\Gamma(E,F)$ denote the family of curves joining $E$ to $F$.
	
	\begin{thm} \label{thm:example}
	There is a Ahlfors $3$-regular, linearly locally connected metric space $X$ homeomorphic to $\mathbb{R}^3$ containing a nondegenerate continuum $E$ with zero capacity in the sense that $\Mod_3 \Gamma_E = 0$. In particular, if $F$ is any other continuum in $X$ disjoint from $E$, then $\Mod_3 \Gamma(E,F)=0$.  
	\end{thm}
	
	As a consequence, the metric space $X$ in \Cref{thm:example} is not quasiconformally equivalent to $\mathbb{R}^3$. Moreover, $X$ does not admit even a weakly quasiconformal parametrization by $\mathbb{R}^3$ in the sense of \cite{NR:22}. The question of whether a metric space as in \Cref{thm:example} exists was raised by the author and Ntalampekos in Remark 7.3 of \cite{NR:21}.  
	
	If one ignores the requirement for $X$ to be Ahlfors $3$-regular, then the example used to prove \Cref{thm:example} is simple to describe. Let $Y$ be a surface of diameter roughly equal to $1$ containing a narrow cusp at a point $x$. Take $X = Y \times (-2,2)$, equipped with the Euclidean product metric. Then $E = \{x\} \times [-1,1]$ is a continuum satisfying the theorem. The space $X$ used to prove \Cref{thm:example} modifies this construction by adding additional area to $Y$ to ensure that $X$ is Ahflors $3$-regular. The construction can be realized as a hypersurface in $\mathbb{R}^4$ that is smooth except on the set $E$, although for simplicity we give an intrinsic geometric construction. 
	
	To put our example in perspective, we recall in a slightly simplified form the following standard definition in the field of analysis on metric spaces; see Definition 3.1 in \cite{HK:98}.
	Let $X$ be a metric space of real Hausdorff dimension $n>1$, which we consider as a metric measure space by equipping it with $n$-dimensional Hausdorff measure. We say that $X$ is \textit{Loewner} if there is a function $\varphi \colon (0,\infty) \to (0, \infty)$ such that $\Mod_n \Gamma(E,F) \geq \varphi(\Delta(E,F))$ for all disjoint, nondegenerate continua $E,F \subset X$, where $\Delta(E,F)$ is the relative distance \[\Delta(E,F) = \frac{\dist(E,F)}{\min\{\diam(E), \diam(F) \}}.\]
	The Loewner property guarantees a large supply of rectifiable curves in a quantitative way. Euclidean space $\mathbb{R}^n$ for all $n \geq 2$, as well as sufficiently nice domains in $\mathbb{R}^n$, are Loewner. Much of the theory of analysis on metric spaces depends on the Loewner property, often together with Ahlfors $n$-regularity for the same value $n$. For example, a theory of quasiconformal mappings very similar to the Euclidean theory can be developed in such spaces \cite{HKST:01}. The Loewner property is essentially equivalent to the statement that $X$ satisfies a $(1,n)$-Poincar\'e inequality as defined in \cite{HK:98}; see Corollary 5.13 in \cite{HK:98}. It is also related to the twin properties of linear local contractibilility and linear local connectivity, which we review in \Cref{sec:background}.
	It is natural to ask to what extent the Loewner property can be relaxed while retaining some reasonably good geometric behavior. However, in light of \Cref{thm:example}, we are not aware of a natural candidate geometric assumption weaker than the Loewner property that could be sufficient to guarantee positive capacity of continua. 
	
	There are a number of related constructions in the literature. Ahlfors $3$-regular metric spaces $X$ homeomorphic to $\mathbb{R}^3$ that are not quasiconformally equivalent to $\mathbb{R}^3$, with the additional property of being a Loewner space, were constructed by Semmes in \cite{Sem:96a}. These examples are metric versions of various classical constructions in geometric topology. Semmes's work has been extended, for instance, in \cite{HW:10,PV:17,PW:14}. Together, the various examples show how the uniformization problem in dimensions three and greater is not expected to have a simple resolution.  
	
	\subsection*{Acknowledgment} I thank Toni Ikonen, Dimitrios Ntalampekos and Kai Rajala for discussions and feedback related to this paper.
	
	\smallskip
	
	\section{Background} \label{sec:background}
	
	In this section, we review the main notation and definitions and briefly discuss the related literature. Let $(X,d)$ be a metric space. We make $X$ into a metric measure space by equipping it with the Hausdorff $n$-measure for suitable choice of real number $n \geq 1$, denoted by $\mathcal{H}^n$. The metric space $X$ is \textit{Ahlfors $n$-regular} if there is a constant $C\geq 1$ such that the Hausdorff $n$-measure satisfies
	\[\frac{r^n}{C} \leq \mathcal{H}^n(B(x,r)) \leq Cr^n \]
	for all $x \in X$ and $r \in (0,\diam(X))$. Here, $B(x,r)$ denotes the open ball of radius $r$ centered at $x$.
	
	Given a family $\Gamma$ of curves in $X$, a Borel function $\rho\colon  X \to [0, \infty]$ is \textit{admissible} for $\Gamma$ if
	\[\int_\gamma \rho\,ds \geq 1 \]
    for all locally rectifiable curves $\gamma \in \Gamma$. For all $p \geq 1$, the \textit{$p$-modulus} of $\Gamma$ is defined as
    \[\Mod_p \Gamma = \inf \int_X \rho^p\,d\mathcal{H}^n, \]
    where the infimum is taken over all functions $\rho$ that are admissible for $\Gamma$. If $X$ is a metric $n$-manifold with locally finite Hausdorff $n$-measure, then the $n$-modulus of $\Gamma$ is called the \textit{conformal modulus} of $\Gamma$. 
    
    The definition of modulus can be extended to objects other than curves. Very commonly, we are interested in families of codimension-$1$ hypersurfaces separating two disjoint continua. Following \cite{EBPC:21}, for two disjoint continua $E,F \subset X$, let $\Sigma(E,F)$ denote the set of boundaries of open sets $U$ such that $E \subset U$ and $F \subset X \setminus \overline{U}$. A Borel function $\rho\colon X \to [0, \infty]$ is \textit{admissible} for $\Sigma(E,F)$ if 
    \[\int_{\partial U} \rho\, d\mathcal{H}^{n-1} \geq 1\]
    for any such open set $U$ whenever $\partial U$ has finite Hausdorff $(n-1)$-measure. The $p$-modulus of $\Sigma(E,F)$, denoted by $\Mod_p \Sigma(E,F)$, is then defined in the same way as for a curve family. 
    
    The following duality result is due to Eriksson-Bique and Poggi-Corradini \cite{EBPC:21}.
	
	\begin{thm} \label{thm:ebpc}
	    Let $X$ be a compact metric space of finite Hausdorff $n$-measure for some $n \geq 1$, and let $E,F \subset X$ be disjoint continua. Let $p \in (1, \infty)$ and $q = p/(p-1)$. If $\Mod_p \Gamma(E,F) >0$, then
	    \begin{equation} \label{equ:modulus}
	      \left(\Mod_p \Gamma(E,F)\right)^{1/p} \, \left(\Mod_q \Sigma(E,F)\right)^{1/q} \geq 2v_n/v_{n-1},  
	    \end{equation}
	    where $v_k = \pi^{k/2}/\Gamma(k/2+1)$ and $\Gamma(\cdot)$ is the Gamma function. 
	    
	    If $\Mod_p \Gamma(E,F) =0$, then $\Mod_q \Sigma(E,F) = \infty$. 
	\end{thm}
	
	Of particular note is the case where $n$ is an integer and $p=n$, so that the conclusion in \Cref{thm:ebpc} concerns conformal modulus. \Cref{thm:ebpc} leaves open the possibility that the conformal modulus is zero and hence that \eqref{equ:modulus} holds only in the degenerate sense. Our example in \Cref{thm:example} thus provides a situation where this can happen. See also \cite{JL:20,Loh:21,LR:21} for further results on duality of modulus in metric spaces.
	
	Such an example is not possible in the two-dimensional setting. Let $X$ and $Z$ be metric $n$-manifolds  with locally finite Hausdorff $n$-measure. We say that a map $h\colon Z \to X$ is \textit{weakly $K$-quasiconformal} if it is continuous and surjective, is the uniform limit of homeomorphisms and satisfies the modulus inequality
	\begin{equation} \label{equ:weak_qc}
	  \Mod_n h(\Gamma) \geq K \Mod_n \Gamma  
	\end{equation}
	for all curve families $\Gamma$ in $Z$. The property of being the uniform limit of homeomorphisms is a natural weakening of being a homeomorphism and is closely related to the property of being a \textit{cell-like mapping}; see Siebenmann \cite{Sie:72} for a precise statement. 
	It is shown by the author and Ntalampekos in \cite{NR:22} that, assuming only that $X$ is a surface with locally finite Hausdorff $2$-measure, such a weakly quasiconformal parametrization from a smooth Riemannian surface always exists. The following theorem builds on a line of increasingly general uniformization theorems for metric surfaces \cite{BK:02,Raj:17,LW:20,MW:21,NR:21}.

	\begin{thm} \label{thm:weak_qc}
	Let $X$ be a metric surface of locally finite Hausdorff $2$-measure. Then there exists a smooth Riemannian surface $Z$ of constant curvature and a weakly quasiconformal map $h \colon Z \to X$ satisfying the inequality
	\[\Mod_2 h(\Gamma) \geq \frac{\pi}{4} \Mod_2 \Gamma \]
	for all curve families $\Gamma$ in $Z$. 
	\end{thm}
	Let $X,Y$ be metric $n$-manifolds with locally finite Hausdorff $n$-measure. A homeomorphism $f \colon X \to Y$ is \textit{quasiconformal} is there exists $K \geq 1$ such that 
	\[K^{-1} \Mod_n \Gamma \leq \Mod_n h(\Gamma) \leq K \Mod_n \Gamma\] for every curve family $\Gamma$ in $X$. If the space $X$ in \Cref{thm:weak_qc} satisfies a geometric condition called \textit{reciprocity} (see Definition 1.3 in \cite{Raj:17}), then the map $h$ given in this theorem is a quasiconformal homeomorphism. Indeed, $h$ satisfies the two-sided modulus inequality
	\[\frac{\pi}{4} \Mod_2 \Gamma \leq \Mod_2 h(\Gamma) \leq 2 \Mod_2 \Gamma \]
	for all curve families $\Gamma$ in $X$. In summary, the only way that a surface may fail to be quasiconformal equivalent to a smooth Riemannian surface is if it has ``too many'' rectifiable curves, which occurs if the surface is too squeezed around a set of small Hausdorff measure.
	
	Finally, the metric space $X$ is \textit{linearly locally connected} if there exists $\lambda \geq 1$ with the following two properties:
	\begin{enumerate}[label=(\alph*)]
	    \item \label{item:llca} For all $x \in X$ and $r>0$, any two points $y,z$ in $B(x,r)$ can be connected by a continuum in $B(x,\lambda R)$. 
	    \item \label{item:llcb} For all $x \in X$ and $r>0$, and two points $y,z$ in $X \setminus B(x,r)$ can be connected by a continuum in $B(x,R/\lambda)$. 
	\end{enumerate} 
	The space $X$ is \textit{linearly locally contractible} if there exists $\lambda \geq 1$ such that for all $x \in X$, $r \in (0,\diam(X)/\lambda)$, the ball $B(x,r)$ can be contracted to a point inside the side $B(x,\lambda r)$. Let $X$ be an Ahlfors $n$-regular metric $n$-manifold for some $n \geq 2$. Under mild assumptions (for example, if $X$ is closed and orientable), we have the following relationships. First, by a result of Semmes \cite[Theorem B.10]{Sem:96}, if $X$ is linearly locally contractible, then $X$ is Loewner. See also \cite[Theorem 6.11]{HK:98}. Second, if $X$ is Loewner, then $X$ is linearly locally connected \cite[Theorem 3.13]{HK:98}. If $n=2$, then all these properties are equivalent; see Lemma 2.5 in \cite{BK:02}. 
	
	
	\smallskip
	
	\section{The construction}
	
	We now continue with the construction. The space $X$ used to prove \Cref{thm:example} has the form $X = Y \times (-2,2)$, where $Y$ is an Ahlfors $2$-regular surface. The idea is for $Y$ to have a very narrow cusp point, while adding in more area to guarantee Ahlfors $2$-regularity. 
	
	Define the function $f \colon [0,1) \to [0, \infty)$ by $f(t) = t^3/3$. The coefficient in the definition of $f$ is included to guarantee that $f'(t) \leq 1$ for all $t$. 
	Let \[Y_0 = \{(t,y) \in \mathbb{R}^2: 0\leq t < 1, 0 \leq y \leq f(t)\}.\] 
	Thus $Y_0$ is a region in $\mathbb{R}^2$ homeomorphic to a closed half-plane. Let $Y_1$ be the double of $Y_0$. That is, $Y_1$ is the quotient space of two copies of $Y_0$ obtained by gluing the manifold boundary $\{(t,y) \in Y_0: y=0 \text{ or } y = f(t)\}$ of each copy of $Y_1$. Observe that $Y_1$ is a surface homeomorphic to $\mathbb{R}^2$. We identify $Y_0$ with the one of its copies in $Y_1$, which we think of as the top side of $Y_1$. 
	
	We obtain the surface $Y$ by gluing additional pieces into $Y_0$. See \Cref{fig:construction} for an illustration. Let $t_{i}^m= i/2^m$, where $m \geq 1$ and $1 \leq i < 2^m$ and $i,m$ are relatively prime, so that the collection $\{t_i^m\}$ is an enumeration of the dyadic rationals in $(0,1)$. For each dyadic rational $t_i^m$, let $I_i^m$ be the interval $\{t_i^m\} \times [0,\min\{2^{-m},f(t_i^m)\}]$, considered as a subset of $Y_0$. Let $S_i^m$ be the double of the square $[0,2^{-m}]^2$, which we call a \textit{pillowcase}, and let $J_i^m$ be a slit of length $\min\{2^{-m},f(t_i^m)\}$ on the boundary edge $[0,2^{-m}] \times \{0\}$, considered as a subset of $S_i^m$. We glue each pillowcase $S_i^m$ to $Y_0$ along the slits $I_i^m$ and $J_i^m$ to form a surface $Y$. More precisely, associated to each slit $I_i^m$ and $J_i^m$ is a set of prime ends homeomorphic to the circle, denoted respectively by $\widehat{I}_i^m$ and $\widehat{J}_i^m$. Identify an endpoint of $I_i^m$ with one of the endpoints of $J_i^m$, and associate to each point in $I_i^m$ the point in $J_i^m$ at the same distance from the chosen endpoint. This induces a continuous bijection $\varphi_i^m$ from $\widehat{I}_i^m$ to $\widehat{J}_i^m$. Let $\widehat{Y}_0$ be the space $Y_0$ with each slit $I_i^m$ replaced by $\widehat{I}_i^m$, and $\widehat{S}_i^m$ be the space $S_i^m$ with $J_i^m$ replaced by $\widehat{J}_i^m$. Then we glue the sets $\widehat{S}_i^m$ to $\widehat{Y}_0$ along the mapping $\varphi_i^m$ to form the space $Y$. One checks that $Y$ is a topological surface. 
	
	\begin{figure}
	\centering
	\resizebox{3.75in}{1.25in}{
    \begin{tikzpicture}[scale=4]
    \draw[thick] (0,0) to (1,0);
    \draw[thick,smooth,samples=100,domain=0:1]  plot(\x,{(\x)*(\x)*(\x)/3});
    \draw[densely dashed] (1,0) to (1,.3333);
    \draw (.5,0) to (.5,.042);
    \draw (.25,0) to (.25,.005);
    \draw (.75,0) to (.75,.141);
    \draw (.125,0) to (.125,.001953);
    \draw (.375,0) to (.375,.018);
    \draw (.625,0) to (.625,.081);
    \draw (.875,0) to (.875,.125);
    \draw (3*.0625,0) to (3*.0625,.002);
    \draw (5*.0625,0) to (5*.0625,.01);
    \draw (7*.0625,0) to (7*.0625,.0279);
    \draw (9*.0625,0) to (9*.0625,.059);
    \draw (11*.0625,0) to (11*.0625,.0625);
    \draw (13*.0625,0) to (13*.0625,.0625);
    \draw (15*.0625,0) to (15*.0625,.0625);
    \draw (7*.03125,0) to (7*.03125,.00333);
    \draw (9*.03125,0) to (9*.03125,.007);
    \draw (11*.03125,0) to (11*.03125,.0135);
    \draw (13*.03125,0) to (13*.03125,.022);
    \draw (15*.03125,0) to (15*.03125,.03125);
    \draw (17*.03125,0) to (17*.03125,.03125);
    \draw (19*.03125,0) to (19*.03125,.03125);
    \draw (21*.03125,0) to (21*.03125,.03125);
    \draw (23*.03125,0) to (23*.03125,.03125);
    \draw (25*.03125,0) to (25*.03125,.03125);
    \draw (27*.03125,0) to (27*.03125,.03125);
    \draw (29*.03125,0) to (29*.03125,.03125);
    \draw (31*.03125,0) to (31*.03125,.03125);
    \end{tikzpicture}}
    \caption{The surface $Y_0$. At each dyadic rational $i/2^m$ ($i,m$ relatively prime), a doubled square (or \textit{pillowcase}) of side length $2^{-m}$ is attached along the corresponding slit.}
    \label{fig:construction}
\end{figure}
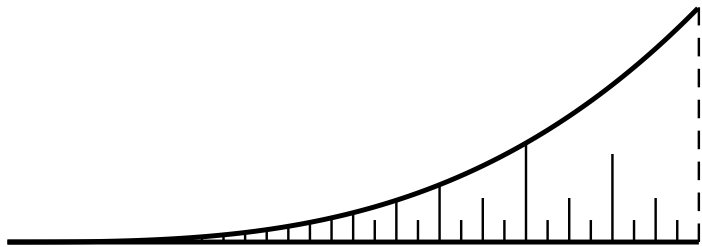
	
	The spaces $\widehat{Y}_0$ and $\widehat{S}_i^m$ each come equipped with a length metric induced by the Euclidean metric on the respective sets. We equip $Y$ with the length metric induced by the gluing, which we denote by $d_Y$. We use $B_Y(x,r)$ to denote the open metric ball in $Y$ centered at $x$ of radius $r$. We then define $X = Y \times (-2,2)$ and equip it with the Euclidean product metric, denoted by $d$. That is, if $z=(z_1,z_2)$ and $w = (w_1,w_2)$ are points in $X$, then $d(z,w) = \sqrt{d_Y(z_1,w_1)^2 + |z_2-w_2|^2}$. Observe that $d$ is also a length metric. Our particular choice of the metric $d$ guarantees that the Hausdorff $3$-measure on $X$ is as nice as possible. Since $Y$ is locally isometric to the Euclidean plane outside a set of Hausdorff $2$-measure zero, we see that $X$ is locally isometric to $\mathbb{R}^3$ outside a set of Hausdorff $3$-measure zero. Thus, for a product set $A = B \times I \subset X$, we have the product relation $\mathcal{H}^3(A) = \mathcal{H}^2(B) \mathcal{H}^1(I)$.

	In the following lemma, we write a point $x \in \widehat{Y}_0$ in coordinates as $x = (x_1,x_2)$, identifying $\widehat{Y}_0$ with the set $Y_0 \subset \mathbb{R}^2$. There is a mild ambiguity in this notation, since the doubled slit points in $\widehat{Y}_0$ share the same coordinate representation, but this should create no confusion.   
	To show that $X$ is Ahlfors $3$-regular, the following lemma is sufficient.
	
	
	\begin{lemm}
	 The surface $Y$ is Ahlfors $2$-regular.
	\end{lemm}
	\begin{proof}
	Let $x \in Y$ and $r \in (0, \diam(Y))$, noting that $\diam(Y) < 2$. Assume first that $x$ belongs to $\widehat{Y}_0$ or one of the pillowcases $S_i^m$. First we establish the lower bound on $\mathcal{H}^2(B_Y(x,r))$. We split into cases.
	
	\begin{case1}
	    Assume that $x \in S_i^m$ for some $i,m$ and $r \leq 2^{-m+1}$. Then $B_Y(x,r)$ intersects $S_i^m$ in a set of area at least $\min\{r^2/2,2^{-2m}\}$, so that $\mathcal{H}^2(B_Y(x,r)) \geq r^2/4$.   
	\end{case1}
	
	\begin{case1}
	    Assume that $x \in Y \setminus (\bigcup S_i^m)$. Let $m \geq 1$ be such that $r/8 < 2^{-m} \leq r/4$. In the first subcase, assume that $x_2 \leq r/4$. There is a dyadic point $t_i^n$, where $n \leq m$, such that $|x_1 - t_i^n| \leq r/4$. By considering the polygonal line from $x$ to $(x_1,0)$ to $(t_i^n,0)$, we see that $\overline{B}(x,r/2)$ intersects $I_i^m$. Observe that $r/4 < 2^{-m+1}$. We now refer back to the previous case for the ball $B((t_i^n,0),r/4)$ to conclude that $\mathcal{H}^2(B_Y(x,r)) \geq r^2/64$.
	    
	    In the second subcase, we have $x_2 > r/4$. Note that the point $x$, considered as a point in $\mathbb{R}^2$, belongs to a square of the form $[i\cdot 2^{-n},(i+1)2^{-n}] \times [2^{-n},2^{-n+1}]$ for some $n$ satisfying $2^{-n} \leq x_2 \leq 2^{-n+1}$ (here, $i$ and $n$ do not need to be relatively prime). This square may not be contained in $Y_0$; nevertheless, we may subdivide this square into eight right triangles, each with legs parallel to the coordinate axes having length $2^{-n-1}$ and hypotenuse running from bottom-left to top-right. Sliding such triangles around the plane, we can find an isosceles right triangle with legs of length $2^{-n-1}$ that contains $x$ and is contained in $Y_0$ and does not cross any of the dyadic slits. This is where we use the property that $f'(t) \leq 1$ for all $t$. 
	    Then $B_Y(x,r)$ intersects $Y_0$ in a set of area at least $2^{-2(n+1)}/2$. 
	    We use the property that $r^2 < 16x_2^2 \leq 64\cdot 2^{-2n}$ to conclude that $\mathcal{H}^2(B_Y(x,r)) \geq r^2/512$.
	\end{case1}
	
	\begin{case1}
	    Assume that $x \in S_i^m$ for some $i,m$ and $r > 2^{-m+1}$. Then $B_Y(x,r/2)$ contains a point $z$ in $Y \setminus (\bigcup S_i^m)$. We refer to the previous case for the ball $B_Y(z,r/2)$ to conclude that $\mathcal{H}^2(B_Y(x,r)) \geq r^2/2048$.
	\end{case1}
	
	Next, we establish the upper bound on $\mathcal{H}^2(B_Y(x,r))$. Again, we split into cases.
	
	\begin{case2}
	Assume that $x \in Y \setminus (\bigcup S_i^m)$. Let $m \geq 0$ be such that $r/2< 2^{-m} \leq r$. The set $\widehat{Y}_0$ contributes at most $4r^2$ to the area of $B_Y(x,r)$, as does the doubled copy of $Y_0$. Next, for a fixed value $m \in \mathbb{N}$, the ball $B_Y(x,r)$ intersects at most $2r\cdot 2^{m}+1$ pillowcases $S_i^m$, each of which has area $2/4^m$. It follows that the sets $S_i^n$, $n \geq m$, that intersect $B_Y(x,r)$ contribute at most \[\sum_{n=m}^\infty (2^{n+1}r+1)\cdot 2/4^n = \sum_{n=m}^\infty 4r/2^n +2/4^n = 8 r \cdot 2^{-m}+ 8\cdot 4^{m}/3 \leq 11r^2\] to the area of $B_Y(x,r)$. Finally, $B_Y(x,r)$ intersects at most four of the sets $S_i^n$ for $n \leq m$, each one of which contributes at most $4r^2$ to the area of $B_Y(x,r)$. We conclude that $\mathcal{H}^2(B_Y(x,r)) \leq 35r^2$.
	\end{case2}
	
	\begin{case2}
	Assume that $x \in S_i^m$ for some $i,m$ and $r>d(x,Y_0)$. Then $B_Y(x,r)$ contains a point $z \in Y \setminus (\bigcup S_i^m)$. So $B_Y(x,r) \subset B_Y(z,2r)$, and we apply the previous case to conclude that $\mathcal{H}^2(B_Y(x,r)) \leq 140r^2$.
	\end{case2}
	
	\begin{case2}
	Assume that $x \in S_i^m$ for some $i,m$ and $r \leq d(x,Y_0)$. Then $B_Y(x,r) \subset S_i^m$, and we have $\mathcal{H}^2(B_Y(x,r)) \leq 4r^2$.
	\end{case2}
	The finishes the case that $x$ is in $\widehat{Y}_0$ or one of the pillowcases $S_i^m$. In the remaining case, $x$ belongs to the doubled copy of $Y_0$. If $B_Y(x,r/2)$ does not intersect $\widehat{Y}_0$, then the metric on $B_Y(x,r/2)$ agrees with the Euclidean metric and we get the Ahlfors regular lower bound. Otherwise, $\widehat{Y}_0$ intersects $\widehat{Y}_0$ in a point $z$ and we apply the $\widehat{Y}_0$ case to $B_Y(z,r/2)$. Similarly, if $B_Y(x,r)$ does not intersect $\widehat{Y}_0$, we immediately get the Ahlfors regular upper bound. Otherwise, $B_Y(x,r)$ intersects $\widehat{Y}_0$ at a point $z$, and we apply the $\widehat{Y}_0$ case to $B_Y(z,2r)$. We conclude that $Y$ is Ahlfors $2$-regular. 
	\end{proof}
	
	We next check the linear local connectedness property.
	
	\begin{lemm}
	    The space $X$ is linearly locally connected.
	\end{lemm}
	\begin{proof}
	Condition \ref{item:llca} in the definition is immediate from the fact that the metric on $X$ is a length metric. Condition \ref{item:llcb} can be established as follows. Let $x \in X$ and $r>0$. If $X \setminus B(x,r)$ is non-empty, we must have $r \leq \diam(X) < 6$. Let $y,z$ be points in $X \setminus B(x,r)$. Write $x$ in components as $x = (x_1,x_2)$, where $x_1 \in Y$ and $x_2\in (-2,2)$. Do similarly for $y,z$. Let $s \in Y$ be such that $d_Y(s,x_1) \geq r/12$. Such an $s$ must exist since $r/12 \leq 1/2< \diam(Y)/2$. Likewise, let $t \in (-2,2)$ be such that $|t-x_2| \geq r/12$. Such a $t$ exists because $r/12<2$.
	
	
	It is immediate that $d_Y(x_1,y_1) \geq r/12$ or $|y_2-x_2| \geq r/12$, and similarly for $z$. There are four cases to consider. If $d_Y(x_1,y_1) \geq r/12$ and $d_Y(x_1,z_1) \geq r/12$, then a piecewise geodesic from $z$ to $(z_1,t)$ to $(y_1,t)$ to $y$ is disjoint from $B(x,r/12)$. If $d_Y(x_1,y_1) \geq r/12$ and $|x_2-z_2| \geq r/12$, then a piecewise geodesic from $z$ to $(y_1,z_2)$ to $y$ is disjoint from $B(x,r/12)$. If $|x_2 - y_2| \geq r/12$ and $|x_2-z_2| \geq r/12$, then a piecewise geodesic from $z$ to $(s,z_2)$ to $(s,y_2)$ to $y$ is disjoint from $B(x,r/12)$. If $|x_2 - y_2| \geq r/12$ and $d_Y(x_1,z_1) \geq r/12$, then a piecewise geodesic from $z$ to $(z_1,y_2)$ to $y$ is disjoint from $B(x,r/12)$. We conclude that condition \ref{item:llcb} holds. 
	\end{proof}
	
	Finally, let $E = \{z\} \times [-1,1]$, where $z = (0,0) \in Y$ is the cusp point. \Cref{thm:example} follows immediately from the following lemma. 
	
	\begin{lemm} \label{lemm:modulus}
	    The family of nontrivial paths intersecting $E$ has $3$-modulus zero.
	\end{lemm}
	\begin{proof}
	For all sufficiently small $\delta,\varepsilon>0$, let $C(\delta,\varepsilon) = B_Y(x,\delta) \times [-1-\varepsilon,1+\varepsilon])$ and let $F(\delta,\varepsilon) = \overline{X \setminus C(\delta,\varepsilon)}$. For a point $x \in Y$, let $\widehat{B}(x,\varepsilon) = B_Y(x,\varepsilon) \cap \widehat{Y}_1$. We decompose $C(\delta,\varepsilon)$ into two parts: $C_1(\delta,\varepsilon) = \widehat{B}(x,\delta) \times [-1-\varepsilon,1+\varepsilon]$ and $C_2(\delta,\varepsilon) = C(\delta, \varepsilon) \setminus C_1(\delta, \varepsilon)$. Since $E$ has $3$-measure zero, the family of nontrivial curves contained in $E$ has $3$-modulus zero. Any other curve must intersect $F(\delta,\varepsilon)$ for some $\delta,\varepsilon>0$. 
	
	Fix $\delta_0,\varepsilon_0>0$. To prove the lemma, it suffices to show that $\Mod_3 \Gamma(E,F(\delta_0,\varepsilon_0)) = 0$. Let $\delta \in (0,\delta_0)$ and $\varepsilon \in (0, \varepsilon_0)$, where $\delta < \varepsilon$. Consider now the function $\rho = \delta^{-1} \chi_{C_1(\delta, \varepsilon)} + \varepsilon^{-1}\chi_{C_2(\delta, \varepsilon)}$. A curve $\gamma \in \Gamma(E,F(\delta,\varepsilon))$ must intersect either $\partial B(x,\delta) \times [-1-\varepsilon,1+\varepsilon]$ or $B(x,\varepsilon) \times \{\pm (1+\varepsilon)\}$. In the first case, the curve $\gamma$ must have length at least $\delta$ in $C_1(\delta,\varepsilon)$. In the second case, $\gamma$ must have length at least $\varepsilon$ in $B(x,\delta) \times ([-1-\varepsilon,-1] \cup [1,1+\varepsilon])$. In both cases, using the property that $\delta < \varepsilon$, we conclude that $\int_\gamma \rho\,ds \geq 1$. Thus the function $\rho$ is admissible for $\Gamma(E,F(\delta,\varepsilon))$, and hence for $\Gamma(E,F(\delta_0,\varepsilon_0))$ as well. 
	
	Observe that $\mathcal{H}^3(C_1(\delta,\varepsilon)) \leq 4(1+\varepsilon)\delta^4$, since $C_1(\delta,\varepsilon)$ is contained in the region $\{(x_1,x_2,t) \in \widehat{Y}_0 \times (-2,2): x_1 \leq \delta, x_2 \leq f(x_1), |t| \leq 1+\varepsilon\}$ and the corresponding region for the doubled copy of $Y_0$. We now compute
	\begin{align*}
	  \int_X \rho^3\, d\mathcal{H}^3 & = \int_{C_1(\delta,\varepsilon)} \frac{1}{\delta^3}\,d\mathcal{H}^2 + \int_{C_2(\delta,\varepsilon)} \frac{1}{\varepsilon^3}\,d\mathcal{H}^3 
	  \\ & = \frac{1}{\delta^3} \mathcal{H}^3(C_1(\delta,\varepsilon)) + \frac{1}{\varepsilon^3} \mathcal{H}^3(C_2(\delta,\varepsilon)) \\
	  & \leq \frac{4(1+\varepsilon)\delta^4}{\delta^3} + \frac{C\delta^2(1+\varepsilon)}{\varepsilon^3},
	\end{align*}
	where $C$ is the Ahlfors $2$-regularity constant for $Y$. Letting $\delta \to 0$ shows that $\Mod_3 \Gamma(E,F(\delta_0,\varepsilon_0)) = 0$. This completes the proof. 
	\end{proof}
	
	\smallskip
	
	\section{Concluding remarks}
	
	One consequence of \Cref{thm:example} is that a weakly quasiconformal map that preserves conformal modulus need not be a homeomorphism. Define a new metric space $\widetilde{X}$ by collapsing the set $E$ to a point $[E]$ and giving $\widetilde{X}$ the corresponding length metric. Equivalently, $\widetilde{X}$ is the metric space determined by the length element $\chi_{X \setminus E}$ on $X$. The space $\widetilde{X}$ is homeomorphic to $X$, which can be seen directly or by \cite[p. 3]{Dav:86}. There is a natural quotient map $h\colon X \to \widetilde{X}$ that collapses $E$ to the equivalence class $[E]$ in $\widetilde{X}$ and is the identity elsewhere. Note that $h$ is not a homeomorphism, though it is the uniform limit of homeomorphisms. 
	
	\begin{prop}
	    The quotient map $h\colon X \to \widetilde{X}$ preserves conformal modulus. 
	\end{prop}
	\begin{proof}
	Let $\Gamma$ be an arbitrary curve family in $X$. Then $\Gamma$ is the union of two families $\Gamma_E$ and $\Gamma \setminus \Gamma_E$, where $\Gamma_E$ is the family of nontrivial curves in $X$ intersecting $E$. Observe that $h$ is locally an isometry at every image point of every curve in $\Gamma \setminus \Gamma_E$. It follows that $\Mod_3 \Gamma \setminus \Gamma_E = \Mod_3 h(\Gamma \setminus \Gamma_E)$. Moreover, by \Cref{lemm:modulus}, $\Mod_3 \Gamma_E = 0$. In addition, the proof of \Cref{lemm:modulus} shows that $\Mod_3 h(\Gamma_E) = 0$ as well. We conclude that $h$ preserves conformal modulus.
	\end{proof}
	
	The original motivation for this note concerns the equivalence of definitions of weak quasiconformality. Let $X,Y$ be separable spaces with locally finite Hausdorff $n$-measure, where $n>1$ is a real number. According to a theorem of Williams \cite{Wil:12}, a homeomorphism $f \colon X \to Y$ satisfies inequality \eqref{equ:weak_qc} for all curve families $\Gamma$ in $X$ (the ``geometric definition'' of quasiconformality) if and only if belongs to the metric Sobolev space $N^{1,p}(X,Y)$ and satisfies the distortion inequality $g_f(x)^2 \leq K J_f(x)$ for a.e. $x \in X$ (the ``analytic definition'' of quasiconformality). Here, $g_f$ is a representative of the minimal weak upper gradient and $J_f$ is the Jacobian of $f$. It was shown in \cite{NR:21} that in dimension $2$ the same equivalence is true for weakly quasiconformal maps between metric surfaces, which are not necessarily homeomorphisms. The argument also applies to metric manifolds of higher dimension if it can be shown that the set of points in $Y$ whose preimage contains more than one point has measure zero. A sufficient condition for this is the statement that $\Mod(E,F)>0$ for all disjoint nondegenerate continua $E,F$; see Lemma 7.8 in \cite{NR:21}. However, as shown by \Cref{thm:example}, this statement does not always hold. Thus the following question remains open.
	\begin{ques}
	    Let $X,Y$ be metric $n$-manifolds of locally finite Hausdorff $n$-measure. 
	    For any weakly quasiconformal map $f\colon X \to Y$, is it necessarily the case that $\{y \in Y: \#f^{-1}(y) >1 \}$ has Hausdorff $n$-measure measure zero?
	\end{ques}

	\bibliographystyle{abbrv}  
	\bibliography{biblio}

\end{document}